\documentclass[11pt,letterpaper]{amsart}

\usepackage[T1]{fontenc}
\usepackage[utf8]{inputenc}
\usepackage{amsmath,amssymb,amsthm,mathtools}
\usepackage{mathrsfs}
\usepackage{enumitem}
\usepackage[colorlinks=true,linkcolor=blue,citecolor=blue,urlcolor=blue]{hyperref}

\makeatletter
\@namedef{subjclassname@2020}{\textup{2020} Mathematics Subject Classification}
\makeatother

\mathtoolsset{showonlyrefs=true}

\newtheorem{theorem}{Theorem}[section]
\newtheorem{theoremletter}{Theorem}
\newtheorem{proposition}[theorem]{Proposition}
\newtheorem{lemma}[theorem]{Lemma}

\theoremstyle{definition}

\newtheorem*{conjecture}{Conjecture}

\theoremstyle{remark}
\newtheorem{remark}[theorem]{Remark}

\numberwithin{equation}{section}

\newcommand{\Sph}{\mathbb S}

\newcommand{\vol}{\operatorname{Vol}}
\newcommand{\area}{\operatorname{Area}}
\newcommand{\Ric}{\operatorname{Ric}}

\newcommand{\wt}{\widetilde}

\newcommand{\barX}{\overline X}

\title[]{A sharp relative comparison inequality for conformal fill-ins of Poincar\'e--Einstein manifolds}
\author[N. Wu]{Nan Wu}
\address[N. Wu]{School of Mathematics \& IMS \\ Nanjing University\\ Nanjing \\ 210093 \\ P.R. China}
\email{nwu2@nju.edu.cn}


\date{\today}

\begin{document}
\begin{abstract}
Let $(X^{n+1},g_+)$ be a Poincar\'e--Einstein manifold with conformal infinity $(M^n,[h])$ of positive Yamabe constant $Y(M,[h])$. Then a relative comparison inequality
\[
 \frac{Y_1(X,M,[\bar g])}{Y_1(\Sph^{n+1}_+,\Sph^n,[g_{\Sph_+^{n+1}}])}
 \ge
 \left(\frac{Y(M,[h])}{Y(\Sph^n,[g_{\Sph^n}])}\right)^{\frac{n}{n+1}}
\]
 holds for the \textbf{type-I} Escobar--Yamabe compactification $Y_1(X,M,[\bar g]$, with equality if and only if $\left(X, g_{+}\right)$ is isometric to the hyperbolic space $\left(\mathbb{H}^{n+1}, g_{\mathbb{H}^{n+1}}\right)$. This confirms a conjecture raised by Sun-Yung A. Chang.
\end{abstract}

\maketitle

\tableofcontents

\section{Introduction}\label{Sect:Introduction}

The conformal fill-in problem asks whether a prescribed conformal manifold can occur as the conformal infinity of a Poincar\'e--Einstein manifold.  Let $X^{n+1}$ be the interior of a smooth compact manifold $\barX$ with boundary $M^n=\partial X$. A complete metric $g_+$ on $X$ is called Poincar\'e--Einstein if
$\Ric_{g_+}=-ng_+$, and if there is a defining function $\rho$ for $M$ such that
$\bar g=\rho^2g_+$ extends smoothly to $\barX$.  The conformal class
\[
  [h]:=[\bar g|_{TM}]
\]
is independent of the choice of defining functions and is called the conformal infinity of $g_+$.  The global fill-in problem is the following: given a compact manifold $X$ with boundary $M$ and a conformal class $[h]$ on $M$, does there exist a Poincar\'e-Einstein metric $g_+$ on $X$ whose conformal infinity is $[h]$?

This problem has both local and global aspects. Near the boundary, the Fefferman--Graham expansion \cite{Fefferman1985conformal} describes the asymptotic structure of a conformally compact Einstein metric.  In a geodesic normal form, we write
\[
  g_+=r^{-2}(dr^2+h_r),
\]
and the Einstein equation determines the Taylor expansion of $h_r$ recursively, up to the familiar formally undetermined transverse-traceless term at order $n$.
If $M$ is real analytic, then there always exists a
Poincaré-Einstein  metric defined on a collar neighborhood $M\times(0, 1]$. This was proved when $M$ is odd-dimensional by Fefferman and Graham \cite{Fefferman1985conformal}, and in the even-dimensional
case by Kichenassamy \cite{Kichenassamy2004}.
Gursky and Sz\'ekelyhidi \cite{GurskySzekelyhidi} proved the same local collar existence theorem in the smooth setting. These results denmonstrate that the local boundary problem is flexible. 

By contrast, the global fill-in problem is substantially more rigid. There are several positive results in this direction.  Graham and Lee proved that conformal classes sufficiently close to the round class on $\Sph^n$ are filled by Poincar\'e--Einstein metrics on the ball \cite{GrahamLee}.  Lee \cite{LeeFredholm} developed the Fredholm deformation theory for nondegenerate conformally compact Einstein metrics.  In dimension four, Anderson \cite{AndersonTopics} studied the boundary map from the moduli space of conformally compact Einstein metrics to the space of conformal infinities and developed a degree-theoretic approach.  

Nevertheless, global obstructions can occur. A striking theorem of Gursky and Han \cite{GurskyHan}, motivated in part by Witten \cite{Witten}, states that there exist infinitely many conformal classes on $\Sph^7$ with positive Yamabe constant that cannot be realized as conformal infinities of Poincar\'e--Einstein metrics in $B^8$. Such a result is especially significant because the obstructed conformal classes have positive Yamabe constant. So, the positivity of Yamabe constant of the conformal infinity is in general not sufficient for the global fill-in. Gursky--Han--Stolz  \cite{GurskyHanStolz} later introduced a spin-cobordism obstruction for conformally compact Einstein fill-ins in dimensions $4k$, $k\in \mathbb{N}$.  There are many related works on existence, compactness, renormalized volume, and comparison results for the conformal fill-in problem, see \cite{ChangGeCompactness,ChangGefilling,ChangGeQing,Chang_2024,Chang2025Poincare,LiQingShi,lee2025rigidity,qing2003rigidity,WangWang,wang2025sharp,wang2023note,wang2026lower} and references therein.

One among important consequences of Gursky and Han \cite{GurskyHan} is the inequality relating the Yamabe constant of conformal infinity and \textbf{type-I} Escobar-Yamabe constant of the compactification.  Such an inequality gives a necessary condition for the existence of fill-ins and is sharp on the hyperbolic ball. To be precise, let $\barX^N$ be a smooth compact manifold with nonempty boundary $M^n$, with $N=n+1$.  For a smooth metric $\bar g$ on $\barX$, define the \textbf{type-I} Escobar--Yamabe constant by
\begin{equation}\label{eq:Y1-def}
 Y_1(\barX,M,[\bar g])
 =\inf_{\wt g\in[\bar g]}
 \frac{\displaystyle\int_XR_{\wt g}\,dv_{\wt g}+2\int_MH_{\wt g}\,d\sigma_{\wt g}}
 {\vol_{\wt g}(X)^{\frac{N-2}{N}}}.
\end{equation}
The Yamabe constant for the conformal infinity is defined by
\begin{equation}
Y(M,[h])=\inf _{\hat{h} \in[h]} \frac{\int_{M} R_{\hat{h}} \mathrm{~d} \sigma_{\hat{h}}}{\left(\int_{M} \mathrm{~d} \sigma_{\hat{h}}\right)^{\frac{n-2}{n}}}.
\end{equation}
We also introduce the isoperimetric ratio
\begin{equation}
I(\barX, M, \bar{g})=\frac{\operatorname{Vol}(M, \gamma)^{N}}{\operatorname{Vol}(\barX, \bar{g})^{N-1}},
\end{equation}
where $\gamma=\left.\bar{g}\right|_M$. In particular, 
$$Y(\Sph^n)=n(n-1)\vol(\Sph^n)^{2/n},\quad Y_1(\Sph^{N}_+,\Sph^n,[g_{\Sph_+^{N}}])=N(N-1)\vol(\Sph^{N}_+)^{\frac{2}{N}}.$$

 Gursky and Han \cite{GurskyHan} established the following inequality.
\begin{theoremletter}[Gursky-Han]\label{thm:Gursky Han}
Let $\left(X^N, g_{+}\right)$be a Poincaré-Einstein manifold satisfying solvability hypotheses for \textbf{type-I} Escobar-Yamabe compactification. Let $(X, M, \bar{g})$ denote its Escobar-Yamabe compactification, and let $\gamma=\left.\bar{g}\right|_M$ be the induced metric on $M=\partial X$.
If the dimension $N \geq 4$, then
\begin{equation}\label{eq:Gursky Han 1}
Y_1(\barX, M,[\bar{g}]) \cdot I(\barX, M, \bar{g})^{\frac{2}{N(N-1)}} \geq \frac{N}{N-2} Y(M,[\gamma]) .
\end{equation}
If the dimension $N=3$, then
\begin{equation}\label{eq:Gursky Han 2}
Y_1(\barX, M,[\bar{g}]) \cdot I(\barX, M, \bar{g})^{1 / 3} \geq 12 \pi \chi(M). 
\end{equation}
If equality occurs, then $\bar{g}$ is Einstein and $\gamma$ has constant scalar curvature.
\end{theoremletter}
The  rigidity part was further proved by Chen--Lai--Wang
\cite{ChenLaiWang} that equality holds in
\eqref{eq:Gursky Han 1} and \eqref{eq:Gursky Han 2} if and only if
\((X,g_+)\) is isometric to 
\((\mathbb H^N,g_{\mathbb H^N})\).

\begin{remark}\label{rmk:main}
In \cite{GurskyHan}, an extra solvability assumption is imposed for
\textbf{type-I} Escobar--Yamabe compactification. This assumption would be 
removed, in combination of  the influential works on  the boundary Yamabe problem by Escobar \cite{Escobar}  and  Brendle--Chen
\cite{BrendleChen}, and the most recent work of Brendle--Wang \cite{BrendleWang} on the positive mass theorem in arbitrary dimensions.
\end{remark}

Using the inequalities \eqref{eq:Gursky Han 1} and
\eqref{eq:Gursky Han 2} in Gursky--Han \cite{GurskyHan} and
a sharp conformally invariant inequality for \textbf{type-II}
Escobar--Yamabe compactification as in
\cite{ChenLaiWang,WangWang}, Alice Chang obtained a non-sharp inequality relating $Y_1(\barX,M,[\bar g])$ and $Y(M,[h]$, and then proposed the following conjecture concerning another sharp conformally invariant inequality.
\begin{conjecture}[Chang Conjecture]
Let $(X^N,g_+)$, $N=n+1\ge 3$ be a smooth Poincar\'e--Einstein manifold with conformal infinity $(M^n,[h])$ such that $Y(M,[h])>0$. Then,
\begin{equation}\label{eq:main-conjecture}
\frac{Y_1(\barX,M,[\bar g])}{Y_1(\Sph^{n+1}_+,\Sph^n,[g_{\Sph_+^{n+1}}])}
 \ge
 \left(\frac{Y(M,[h])}{Y(\Sph^n,[g_{\Sph^n}])}\right)^{\frac{n}{n+1}},
\end{equation}    
with equality if and only if $\left(X^{N}, g_{+}\right)$ is isometric to $\left(\mathbb{H}^N, g_{\mathbb{H}^N}\right)$. 
\end{conjecture}
Chang and Ge \cite{ChangGefilling} studied the existence and compactness problem for the conformal fill-ins in dimension four and obtained the same type estimates for all dimensions $N \geq 4$ with non-sharp constants. 

Our contribution is to give a positive answer to the above conjecture.
\begin{theorem}\label{thm:main}
The Chang conjecture is true.
\end{theorem}

We describe the strategy of the proof of Theorem \ref{thm:main} while demonstrating the relation with Gursky--Han \cite{GurskyHan}.  Throughout the paper, let $g=\rho^2g_+$ be the \textbf{type-I} Escobar-Yamabe compactification, normalized as
\[
  R_g=N(N-1),\qquad H_g=0.
\]
Set
\[
  V=\vol_g(X),\qquad A=\vol_g(M),\qquad \gamma=g|_{TM}.
\]
Then,
\[
  Y_1(\barX,M,[g])=N(N-1)V^{2/N}.
\]
Gursky and Han's identity \cite[(2.15)]{GurskyHan} leads to
\begin{equation}\label{eq:GH-identity}
 n(n-1)A =\int_X R_g dv_g=\int_M R_\gamma\,d\sigma_\gamma+\frac{2}{n-1}\int_X\rho |E_g|^2\,dv_g,
\end{equation}
where
\[E_g=\Ric_g-(N-1)g\]
is the trace-free Ricci tensor.  Then Gursky and Han's inequality follows by dropping the nonnegative error term in \eqref{eq:GH-identity}.  

The relative comparison inequality  \eqref{eq:main-conjecture} is equivalent to
\begin{equation}\label{eq:Ineq equi}
Y(M,[\gamma]) \leq n(n-1) \left(C_NV\right)^{\frac{2}{n}} \qquad \mathrm{with}\quad C_N=\frac{\vol(\Sph^n)}{\vol(\Sph^{N}_+)}.
\end{equation}
By definition of the Yamabe constant, our primary objective is to establish
\begin{equation}\label{Ineq:key}
\left(\frac{C_NV}{A}\right)^{2/n}
 \ge
 \frac{\int_MR_\gamma\,d\sigma_\gamma}{n(n-1)A}.
 \end{equation}
This corresponds to Propositions \ref{lem:N3-compensated}-\ref{prop:power-bridge}.
The conformal Einstein equations lead to (see Lemma \ref{lem:beta-positive})
\[
  A\ge C_N V.
\]
The important ingredient in our proof is that the excess can be controlled
by the error term in \eqref{eq:GH-identity}:
\[
  Q=\int_X \rho |E_g|^2\,dv_g.
\]

The proof relies on one important quantity related to the defining function:
\[
  w=-(\Delta_g\rho+N\rho).
\]
The conformal Einstein equations guarantee
\begin{equation}\label{eq: Delta w}
-\Delta_g w=\frac{N}{(N-2)^2}\rho |E_g|^2 \qquad \mathrm{in~~} X,
\end{equation}
coupled with $w=0$ on $M$. Then the maximum principle applies to show that $w\ge 0, ~0 \leq \rho \leq 1$ in $X$.
On the hemisphere $(\Sph_+^N,d s^2+\cos ^2 s g_{\mathbb{S}^{N-1}})$ for $s\in [0, \pi/2]$, the geometry is encoded by the radial parameter function $s=\arcsin r$ and a one-dimensional calibrated function $\Phi_N(s)$ satisfying
\[
 -\Phi_N''+(N-1)\tan s\,\Phi_N'=C_N.
\]
Here, we distinguish its defining function by $r$.
So, $\Phi_N(\arcsin \rho)$ is a natural counterpart for  a general Poincar\'e--Einstein manifold. With this particular defining function $\rho$, the divergence theorem provides a key defect formula
\begin{equation}\label{defect-formula}
  A-C_NV=\int_X w\,\beta_N(\rho)\,dv_g,
\end{equation}
where $\beta_N$ is given in  Lemma \ref{key defect lemma}.

The remainder of the proof is devoted to estimating the defect term on the
right-hand side. For \(N=3,4\), we construct appropriate  supersolutions
\(f(\rho)\) satisfying
$$
  -\Delta_g f(\rho)\ge \beta_N(\rho),
  \qquad
  f(0)=0.
$$
Integration by parts together with \eqref{eq: Delta w} controls the
defect term
$$
  \int_X w\,\beta_N(\rho)\,dv_g
$$
by a constant multiple of
$$
  \int_X \rho |E_g|^2\,dv_g.
$$
For \(N\ge 5\), we instead use a nonlinear correction term to prove the
 estimate
$$
  \left(\frac{C_N V}{A}\right)^{2/n}
  \ge
  \frac{\int_M R_\gamma\,d\sigma_\gamma}{n(n-1)A}.
$$
Therefore, for all dimensions $N \geq 3$, we obtain \eqref{Ineq:key}, which concludes the proof of the inequality \eqref{eq:main-conjecture}.

For the rigidity part, when equality holds for \eqref{eq:main-conjecture}, $E_g=0$. By \eqref{eq: Delta w}, $w$ is  harmonic  in $X$ with $w|_M=0$. Hence $w\equiv0$.  Conformal Einstein equations  implies that $\rho$ satisfies the Obata equation, which yields the rigidity.

The paper is organized as follows.  In Section \ref{sec:prelim}, we establish basic properties of the defining function and the important quantity $w$ by conformal Einstein equations.  In Section \ref{sec:defect}, we derive the defect formula \eqref{defect-formula}.  Section \ref{sec:N3} is to construct appropriate supersolutions related to $\beta_N$ to obtain the defect estimates. Section \ref{sec:main-proof} is devoted to the proof of Theorem \ref{thm:main}.

\medskip
\noindent \textbf{Acknowledgment.}
The author is indebted to Zetian Yan for introducing him to this problem and for many insightful discussions. He would also like to thank his advisor, Qing Han, for the constant support and encouragement, and Alice Chang, Xuezhang Chen, and Matthew Gursky for their valuable comments and suggestions.

\section{Type-I Escobar-Yamabe compactification}\label{sec:prelim}
Let 
\[
  g=\rho^2g_+
\]
be the \textbf{type-I} Escobar-Yamabe compactification as before and $\gamma=g|_{TM}$. 

We first derive some identities related to the traceless Ricci curvature $E_g$ and the important quantity $w$.
\begin{lemma}\label{lem:CE-identities}
Let \begin{equation}\label{eq:E-def}
 E:=E_g=\Ric_g-(N-1)g.
\end{equation}
Define
\begin{equation}\label{eq:w-equations}
 w=-\left(\Delta_g\rho+N\rho\right).
\end{equation}
Then, in $X$,
\begin{align}
        |\nabla\rho|^2
        &=1-\rho^2-\frac{2\rho}{N}w,\label{eq:grad-rho}\\
        -\Delta_g w
        & =\frac{N}{(N-2)^2}\rho |E|^2.\label{eq:Delta-U}
\end{align}
Moreover, 
\[
w=0 \quad \text{ on~~ } M,
\]
and 
\begin{equation}\label{eq:w-positive}
        w\ge0 \quad \text{in~~} \barX.
\end{equation}
Consequently, $0\le \rho\le1$ in $\barX$.
\end{lemma}
\begin{proof}
For simplicity, covariant derivatives and norms are with respect to $g$.
Since $g_{+}$ is Einstein, the conformal change of Ricci curvature gives
\begin{equation}\label{eq:E-hessian}
E=-(N-2) \rho^{-1}\left(\nabla^2 \rho-\frac{\Delta \rho}{N} g\right).
\end{equation}

The conformal transformation of scalar curvature shows 
\begin{equation}\label{eq: scalar conformal}
 -N(N-1)
        =\rho^2N(N-1)+2(N-1)\rho\Delta\rho
         -N(N-1)|\nabla\rho|^2.
\end{equation}
Dividing by $N(N-1)$ in \eqref{eq: scalar conformal} yields
\[
|\nabla \rho|^2=1+\rho^2+\frac{2 \rho}{N} \Delta \rho=1-\rho^2-\frac{2 \rho}{N} w,
\]
which implies \eqref{eq:grad-rho}. 

We next derive the identity for $dw$. Differentiating  \eqref{eq:grad-rho} gives
\begin{equation}\label{eq: hessian rho 1}
 \nabla^2 \rho(\nabla \rho, \cdot)=- \rho d \rho-\frac{1}{N}(w d \rho+\rho d w).
\end{equation}
On the other hand, \eqref{eq:E-hessian} implies
\begin{equation}\label{eq: hessian rho 2}
\begin{aligned}
\nabla^2 \rho(\nabla \rho, \cdot)&=\frac{\Delta_g \rho}{N} d \rho-\frac{\rho}{N-2} E(\nabla \rho, \cdot)\\
&=-\left(\frac{w}{N}+\rho\right) d \rho-\frac{\rho}{N-2} E(\nabla \rho, \cdot).
\end{aligned}
\end{equation}
Comparing the two expressions \eqref{eq: hessian rho 1}, \eqref{eq: hessian rho 2} and dividing by $\rho$ in the interior yield
\begin{equation}\label{eq:grad-U}
        d w=\frac{N}{N-2}E(\nabla\rho,\cdot).
\end{equation}
Since $R_g$ is constant, the contracted Bianchi identity gives $\text{div}(E)=0$. Although (\ref{eq:grad-U}) is first derived in the interior, both sides are smooth up to
the boundary, and hence the identity extends continuously to \(\barX\).

Taking the divergence of \eqref{eq:grad-U} and using \eqref{eq:E-hessian} once more, we obtain
\[
\begin{aligned}
        \Delta w
        &= \frac{N}{N-2}\nabla^i(E_{ij}\rho^j)
         = \frac{N}{N-2}E_{ij}\rho^{ij}\\
        &= \frac{N}{N-2}E^{ij}\left(-\frac{\rho}{N-2}E_{ij}+\frac{\Delta\rho}{N}g_{ij}\right)
         =-\frac{N}{(N-2)^2}\rho |E|^2.
\end{aligned}
\]
 This proves \eqref{eq:Delta-U}.
 Since $E$ is smooth on $\barX$,  \eqref{eq:E-hessian} implies
\begin{equation}\label{eq: obata rho}
 \nabla^2\rho-\frac{\Delta\rho}{N}g=0
        \qquad\text{on }M.
\end{equation}
Let $\left\{e_\alpha\right\}_{\alpha=1}^n$ be a local orthonormal frame tangent to $M$. Since $\rho=0$ and $\nabla \rho=-\nu$ on $M$, where $\nu$ is the outward unit normal, we have
\[
\nabla^2 \rho\left(e_\alpha, e_\beta\right)=-L_{\alpha \beta},
\]
where $L$ is the second fundamental form of $M \subset (\barX, g)$. Taking the tangential trace in \eqref{eq: obata rho} and using $H_g=0$, we obtain
\[
0=\sum_{\alpha=1}^n\left(\nabla^2 \rho\left(e_\alpha, e_\alpha\right)-\frac{\Delta \rho}{N}\right)=-H_g-\frac{n}{N} \Delta \rho=-\frac{n}{N} \Delta \rho.
\]
This gives $\Delta \rho=0$ on $M$, whence $w=-(\Delta \rho+N \rho)=0$ on $M$. This in turn implies $L_{\alpha \beta}=0$, so the boundary $M$ is totally geodesic.

The maximum principle applied to \eqref{eq:Delta-U} with $\left.w\right|_M=0$, gives $w \geq 0$ on $\barX$. 
Finally, if $\rho$ admits an interior maximum point with value larger than $1$, then at that point, \eqref{eq:grad-rho} implies
\[
        0=|\nabla\rho|^2\le1-\rho^2<0.
\]
A contradiction! Since $\rho$ is a defining function, we obtain $0 \le \rho \le 1$.
\end{proof}

\section{The area-volume defect and extension of correction functions}\label{sec:defect}
We use the same notations in Section \ref{sec:prelim}. Let $g=\rho^2 g_{+}$ be the \textbf{type-I} Escobar-Yamabe compactification, $V=\operatorname{Vol}_g(X), A=\operatorname{Vol}_g(M)$, and
\[w=-\left(\Delta_g \rho+N \rho\right) .\]
By  Lemma \ref{lem:CE-identities} we know that $0 \leq \rho \leq 1, w \geq 0$ in $\overline X$, and
\[|\nabla \rho|^2=1-\rho^2-\frac{2 \rho}{N} w, \quad \Delta_g \rho=-N \rho-w.\]

Recall from \eqref{eq:Ineq equi} that
\begin{equation}\label{eq:C-N}
 C_N=\frac{\vol(\Sph^n)}{\vol(\Sph^{N}_+)} \quad \Rightarrow\quad 
 C_N^{-1}=\int_0^{\pi/2}\cos^{N-1}t\,dt.
\end{equation}

For $0\le s<\pi/2$, we define $\Phi_N$ through $\Phi_N(0)=0$ and
\begin{equation}\label{eq:Phi-def}
 \Phi_N'(s)=C_N\frac{\displaystyle\int_s^{\pi/2}\cos^{N-1}t\,dt}{\cos^{N-1}s}.
\end{equation}
Then $\Phi_N^{\prime}(0)=1$, and differentiating the above equation yields
\begin{equation}\label{eq:Phi-ODE}
 -\Phi_N''(s)+(N-1)\tan s\,\Phi_N'(s)=C_N.
\end{equation}
The following is the geometric meaning of $\Phi_N$: In the coordinates 
\[g_{\mathbb{S}_+^N}=d s^2+\cos ^2 s g_{\mathbb{S}^{N-1}}, \quad 0 \leq s \leq \frac{\pi}{2},\]
the hypersurface $s=0$ is the equator and $\nu=-\partial_s$ is the outward unit normal on $\partial \Sph_{+}^N$. We distinguish the defining function by $r=\sin s$ and $g_+=r^{-2} g_{\Sph_+^N}$ is the hyperbolic metric. Then $\Phi_N$ satisfies
\begin{equation}
-\Delta_{\Sph_{+}^N} \Phi_N=C_N \mathrm{~~in~~} \Sph_{+}^N, \quad-\partial_\nu \Phi_N=1 \text { on } \partial \Sph_{+}^N, \quad \Phi_N^{\prime}(\pi / 2)=0.
\end{equation}
More importantly, the vector field
\begin{equation}
X_N:=-\nabla \Phi_N
\end{equation}
satisfies 
\begin{equation}
\operatorname{div} X_N=C_N, \quad X_N \cdot \nu=1 \quad \text { on } \partial \Sph_{+}^N
\end{equation}
such that
\begin{equation}
\operatorname{Area}\left(\Sph^{N-1}\right)=\int_{\Sph_{+}^N} \operatorname{div} X_N=C_N \operatorname{Vol}\left(\Sph_{+}^N\right).
\end{equation}
So, $-\nabla\Phi_N$ is the spherical analogue of the Euclidean calibration vector field $X=x\in \mathbb{R}^N$ on the unit ball $\mathbb{B}^N \subset \mathbb{R}^N$, for which
\begin{equation}
\operatorname{div} X=N, \quad X \cdot \nu=1.
\end{equation}

Equation \eqref{eq:Phi-ODE} indicates that  $\Phi_N(s)$ with $s=\arcsin r$ is not a priori smooth at points where $r=1$ because  $(\arcsin r)'$ is singular there. 

Define
\begin{equation}\label{eq:K-N}
 K_N(r)=\int_r^1(1-t^2)^{\frac{N-2}{2}}\,dt, \qquad r \in [0,1].
\end{equation}
Actually, \(K_N\) comes from the change of variables \(r=\sin s\):
\[
  K_N(\sin s)
  =
  \int_s^{\pi/2}\cos^{N-1}t\,dt.
\]

For future citations, we need the following extension lemma for several functions related to $\Phi_N$.
\begin{lemma}\label{lem:endpoint-regularity}
For $N\ge3$,  let $a=1-r^2$. The function
\begin{equation}\label{eq:def of H-N}
H_N(r)=\frac{K_N(r)}{(1-r^2)^{N/2}},\qquad 0\le r<1,
\end{equation}
extends smoothly to $r=1$ by
\begin{equation}\label{eq:HN-endpoint-values}
 H_N(1)=\frac1N, \qquad H_N(0)=\frac1{C_N}
\end{equation}
and near $r=1$,
\begin{equation}\label{eq:HN-expansion}
 H_N(r)=\frac1N+\frac{a}{2(N+2)}+\frac{3a^2}{8(N+4)}+O(a^3).
\end{equation}
Consequently, the function
\begin{equation}\label{eq:Psi def}
\Psi_N(r)=\Phi_N(\arcsin r),\qquad 0\le r<1,
\end{equation}
extends smoothly to $[0,1]$ and satisfies
\begin{equation}\label{eq:Psi-derivative}
 \Psi_N'(r)=C_NH_N(r),\qquad \Psi_N'(0)=1,
 \qquad \Psi_N'(1)=\frac{C_N}{N},
\end{equation}
and
\begin{equation}\label{eq:Psi-derivative 2}
-\left(1-r^2\right) \Psi_N^{\prime \prime}(r)+N r \Psi_N^{\prime}(r)=C_N .
\end{equation}
Furthermore, the function
\begin{equation}\label{eq:beta-N}
\begin{aligned}
\beta_N(r)&=\Psi_N^{\prime}(r)+\frac{2 r}{N} \Psi_N^{\prime \prime}(r)\\
 &=C_N\frac{(1+r^2)K_N(r)-\frac{2r}{N}(1-r^2)^{N/2}}
 {(1-r^2)^{\frac{N+2}{2}}}, \quad 0\le r<1
\end{aligned}
\end{equation}
extends smoothly to $r=1$, with
\begin{equation}\label{eq:beta-endpoint}
 \beta_N(1)=\frac{C_N}{N+2}.
\end{equation}
Finally, define
\[
 f_N(r)=\frac{C_N}{N}\left(\frac1{C_N}-H_N(r)\right),\qquad
 P_N[f_N](r)=f_N'(r)+\frac{2r}{N}f_N''(r).
\]
Then, $f_N$ and $P_N[f_N]$ extend smoothly to $r=1$, and
\begin{equation}\label{eq:PN-endpoint}
 f_N(1)=\frac1N\left(1-\frac{C_N}{N}\right),\qquad
 P_N[f_N](1)=\frac{C_N(N-2)}{N^2(N+4)}.
\end{equation}
\end{lemma}

\begin{proof}
The only issue is the potential singularity at $r=1$.  With $a=1-r^2$ and the change of variables $u=1-t^2$, we obtain
\[
 K_N(r)=\frac12\int_0^a u^{\frac N2-1}(1-u)^{-1/2}\,du
       =a^{N/2}\,\mathcal H_N(a),
\]
where
\[
 \mathcal H_N(a)=\frac12\int_0^1 y^{\frac N2-1}(1-ay)^{-1/2}\,dy.
\]
Thus the function $\mathcal H_N(a)$ is smooth for $a$ near $0$, and $H_N(r)=\mathcal H_N(1-r^2)$ near $r=1$. Expanding $(1-ay)^{-1/2}$ at $a=0$ gives
\[
 \mathcal H_N(a)=\frac1N+\frac{a}{2(N+2)}+\frac{3a^2}{8(N+4)}+O(a^3),
\]
which is exactly \eqref{eq:HN-expansion}.

The identity \eqref{eq:Psi-derivative} follows from
\[
 \frac{d}{dr}\Phi_N(\arcsin r)
 =\frac{\Phi_N'(s)}{\sqrt{1-r^2}}
 =C_N\frac{K_N(r)}{(1-r^2)^{N/2}}
 =C_NH_N(r).
\]
For \eqref{eq:Psi-derivative 2}, by \eqref{eq:Psi-derivative} we have 
\begin{equation}
\Psi_N'(r)=C_NH_N(r).
\end{equation}
Since
\begin{equation}
K_N^{\prime}(r)=-\left(1-r^2\right)^{\frac{N-2}{2}},
\end{equation}
we have
$$
H_N^{\prime}(r) =-\left(1-r^2\right)^{-1}+N r K_N(r)\left(1-r^2\right)^{-\frac{N}{2}-1}
$$
and then
\begin{equation}\label{eq:Psi 3}
    \Psi_N^{\prime \prime}(r)=C_N H_N^{\prime}(r)=C_N\left[-\frac{1}{1-r^2}+\frac{N r K_N(r)}{\left(1-r^2\right)^{\frac{N}{2}+1}}\right] .
\end{equation}
Hence,
\begin{equation}
\begin{aligned}
-\left(1-r^2\right) \Psi_N^{\prime \prime}(r)+N r \Psi_N^{\prime}(r)&=  -C_N\left(1-r^2\right)\left[-\frac{1}{1-r^2}+\frac{N r K_N(r)}{\left(1-r^2\right)^{\frac{N}{2}+1}}\right] \\
&\qquad +N r C_N \frac{K_N(r)}{\left(1-r^2\right)^{N / 2}} \\
&=  C_N,
\end{aligned}
\end{equation}
which proves \eqref{eq:Psi-derivative 2}.
By \eqref{eq:Psi-derivative} and \eqref{eq:Psi 3} we obtain
\begin{equation}\label{eq:beta N}
\begin{aligned}
\beta_N(r) &=\Psi_N^{\prime}(r)+\frac{2 r}{N} \Psi_N^{\prime \prime}(r)\\
& =C_N \frac{K_N(r)}{\left(1-r^2\right)^{N / 2}}+\frac{2 r}{N} C_N\left[-\frac{1}{1-r^2}+\frac{N r K_N(r)}{\left(1-r^2\right)^{\frac{N}{2}+1}}\right] \\
& =C_N\left[\frac{K_N(r)}{\left(1-r^2\right)^{N / 2}}+\frac{2 r^2 K_N(r)}{\left(1-r^2\right)^{\frac{N}{2}+1}}-\frac{2 r}{N\left(1-r^2\right)}\right] \\
& =C_N\left[\frac{\left(1+r^2\right) K_N(r)}{\left(1-r^2\right)^{\frac{N+2}{2}}}-\frac{2 r\left(1-r^2\right)^{N / 2}}{N\left(1-r^2\right)^{\frac{N+2}{2}}}\right],
\end{aligned}
\end{equation}
which yields \eqref{eq:beta-N}.
Equivalently,
\begin{equation}\label{eq:beta-via-H}
 \beta_N(r)=C_N\frac{(1+r^2)H_N(r)-2r/N}{1-r^2}.
\end{equation}
Using \eqref{eq:HN-expansion} and $r=(1-a)^{1/2}=1-a/2+O(a^2)$, the numerator in \eqref{eq:beta-via-H} is 
\[\left(1+r^2\right) H_N(r)-\frac{2 r}{N}=\frac{a}{N+2}+O\left(a^2\right) .\]
Then $\beta_N$ extends smoothly to $r=1$ by $\beta_N(1)=C_N /(N+2)$, hence \eqref{eq:beta-endpoint} follows. 

Finally, the smooth extension of $f_N$ and $P_N[f_N]$ follows from that of $H_N$. The endpoint value of $f_N$ follows from $H_N(1)=1/N$.  Differentiating \eqref{eq:HN-expansion} gives
\[
 H_N'(1)=-\frac1{N+2},\qquad
 H_N''(1)=\frac{2(N+1)}{(N+2)(N+4)}.
\]
Since $f_N'=-(C_N/N)H_N'$ and $f_N''=-(C_N/N)H_N''$, we obtain
\[\begin{aligned} P_N[f_N](1) & =\frac{C_N}{N(N+2)}-\frac{4 C_N(N+1)}{N^2(N+2)(N+4)} \\ & =\frac{C_N(N-2)}{N^2(N+4)}.\end{aligned}\]
This proves \eqref{eq:PN-endpoint} and completes the proof.
\end{proof}
We now prove the following key defect formula.
\begin{lemma}\label{key defect lemma}
Let $\rho$ be the defining function for \textbf{type-I} Escobar-Yamabe compactification and the function
\[
 \Psi_N(\rho)=\Phi_N(\arcsin \rho),\quad 0\le \rho\le1
\] 
be given in \eqref{eq:Psi def}. Then,
\begin{equation}\label{eq:defect-D}
 A-C_NV=\int_Xw\,\beta_N(\rho)\,dv_g,
\end{equation}
where $\beta_N(\rho)=\Psi_N^{\prime}(\rho)+\frac{2 \rho}{N} \Psi_N^{\prime \prime}(\rho)$ is defined in \eqref{eq:beta-N}.
\end{lemma}
\begin{proof}
By Lemma \ref{lem:endpoint-regularity}, $\Psi_N(\rho)$ is smooth on $[0,1]$. On $M$, we have $\rho=0$, $\Psi_N^{\prime}(0)=1$, and $\nabla \rho=-\nu$, where $\nu$ is the outward unit normal. Hence,
\[
 -\partial_\nu \Psi_N(\rho)=1\qquad\text{on }M.
\]
This together the divergence theorem yields
\begin{equation}\label{eq:area-div}
 A=\int_M 1 d\sigma_\gamma=\int_X-\Delta_g \Psi_N(\rho) d v_g.
\end{equation}
On the open set $\{\rho<1\}$,
the composition $\Psi_N(\rho)$ is smooth on $X$. By the chain rule,
\begin{equation}
\Delta_g \Psi_N(\rho)=\Psi_N^{\prime \prime}(\rho)|\nabla \rho|^2+\Psi_N^{\prime}(\rho) \Delta_g \rho .
\end{equation}

By \eqref{eq:w-equations} and \eqref{eq:grad-rho} we have
\[|\nabla \rho|^2=1-\rho^2-\frac{2 \rho}{N} w, \qquad \Delta_g \rho=-N \rho-w .\]
Hence,
\begin{equation}\label{eq: derive of Psi}
\begin{aligned}
-\Delta_g \Psi_N(\rho)= & -\Psi_N^{\prime \prime}(\rho)\left(1-\rho^2-\frac{2 \rho}{N} w\right)-\Psi_N^{\prime}(\rho)(-N \rho-w) \\
= & {\left[-\left(1-\rho^2\right) \Psi_N^{\prime \prime}(\rho)+N \rho \Psi_N^{\prime}(\rho)\right] } \\
& +w\left[\frac{2 \rho}{N} \Psi_N^{\prime \prime}(\rho)+\Psi_N^{\prime}(\rho)\right].
\end{aligned} 
\end{equation}
Combining \eqref{eq: derive of Psi}, \eqref{eq:Psi-derivative 2} with  \eqref{eq:beta-N} implies
\begin{equation}\label{eq:defect-pointwise}
 -\Delta_g\Psi_N(\rho)-C_N=w\,\beta_N(\rho).
\end{equation}
All functions on both sides of \eqref{eq:defect-pointwise} are continuous on $\barX$ by Lemma \ref{lem:endpoint-regularity}.  Moreover, the set $\{\rho<1\}$ is dense. Otherwise,  $\rho \equiv 1$ on a nonempty open set, and hence $w=-(\Delta\rho+N\rho)=-N<0$ there, contradicting $w\ge0$ by Lemma \ref{lem:CE-identities}.  Hence, \eqref{eq:defect-pointwise}  extends to  $\barX$ by continuity. Integrating \eqref{eq:defect-pointwise} together with \eqref{eq:area-div} yields
\[A-C_N V=\int_X\left(-\Delta_g \Psi_N(\rho)-C_N\right) d v_g=\int_X w \beta_N(\rho) d v_g,\]
which is exactly \eqref{eq:defect-D}.
\end{proof}
The next lemma gives a rough comparison between the boundary area $A$ and the compactified volume $V$. It will be used repeatedly in the sequel.
\begin{lemma}\label{lem:beta-positive}
For every $N\ge3$, the function $\beta_N(r)$ satisfies
\[
 \beta_N(r)>0\qquad \text{ for}\quad  0\le r\le1.
\]
Consequently,
\begin{equation}\label{eq:reverse-bridge-general}
 A\ge C_NV.
\end{equation}
\end{lemma}

\begin{proof}
For $0 \leq r<1$, by \eqref{eq:beta-N} the positivity of $\beta_N$ is equivalent to
\[
 K_N(r)\ge \frac{2r(1-r^2)^{N/2}}{N(1+r^2)}.
\]
Define
$$
F_N(r)=K_N(r)-\frac{2r(1-r^2)^{N/2}}{N(1+r^2)}.
$$
Then $F_N(1)=0$ and
\[
 F_N'(r)=
 -\frac{(1-r^2)^{\frac{N}{2}}}{N(1+r^2)^2}
 \left[N(r^2+1)+2(1-r^2)\right]\le0.
\]
Thus $F_N(r)>0$ for $0\le r<1$.  Together with  
$\beta_N(1)=C_N/(N+2)$ by Lemma \ref{lem:endpoint-regularity}, this proves the positivity of $\beta_N$. Since $w \geq 0$, the defect formula \eqref{eq:defect-D} yields
\[A-C_N V=\int_X w \beta_N(\rho) d v_g \geq 0,\]
and hence $A \geq C_N V.$
\end{proof}

\section{Supersolutions for \texorpdfstring{\(\beta_N\)}{beta_N}
and defect control}\label{sec:N3}
In this section, we construct one-dimensional supersolutions adapted to  $\beta_N$ and use them to control the area-volume defect. We recall the notation from Section \ref{sec:prelim}:
\[\begin{gathered}V=\operatorname{Vol}_g(X), \quad A=\operatorname{Vol}_g(M), \quad Q=\int_X \rho|E|^2 d v_g \\ w=-(\Delta_g \rho+N \rho), \quad-\Delta_g w=\frac{N}{(N-2)^2} \rho|E|^2.\end{gathered}\]
The defect identity \eqref{eq:defect-D} as in Lemma \ref{key defect lemma}  is
\[A-C_N V=\int_X w \beta_N(\rho) d v_g.\]
 For a smooth function $f=f(\rho)$ on $[0, 1]$, we have
\begin{equation}\label{eq: Delta f rho 1}
\begin{aligned}
-\Delta_g f(\rho) & =-f^{\prime \prime}(\rho)|\nabla \rho|^2-f^{\prime}(\rho) \Delta_g \rho \\ & =\left[-\left(1-\rho^2\right) f^{\prime \prime}(\rho)+N \rho f^{\prime}(\rho)\right]+w\left[f^{\prime}(\rho)+\frac{2 \rho}{N} f^{\prime \prime}(\rho)\right].
\end{aligned}
\end{equation}
Accordingly, write
\begin{equation}\label{eq: def of LN PN 1D}
L_N f=-\left(1-r^2\right) f^{\prime \prime}+N r f^{\prime}, \quad P_N[f]=f^{\prime}+\frac{2 r}{N} f^{\prime \prime},
\end{equation}
Then
\begin{equation}\label{eq: Detla f rho main}
-\Delta_g f(\rho)=L_N f(\rho)+w P_N[f](\rho).
\end{equation}

We define a test function $f_N$ by
\begin{equation}\label{eq: def of test function f-N}
    f_N(\rho)=\frac{C_N}{N}\left(\frac{1}{C_N}-H_N(\rho)\right),
\end{equation}
where $H_N$ is given in \eqref{eq:def of H-N}. As we will see later, $f_N$ satisfies
\begin{equation}\label{eq:property of f-N}
    L_Nf_N=\beta_N, \qquad f_N(0)=0,
\end{equation}
where $\beta_N$ is defined by \eqref{eq:beta-N}.
\subsection{The three-dimensional case}
Let $N=3$ and
\[
 C_3=\frac{\area(\Sph^2)}{\vol(\Sph^3_+)}=\frac4\pi.
\]
\begin{proposition}\label{lem:N3-compensated}
For $N=3$, there holds
\begin{equation}\label{eq:N3-compensated}
 \frac{1}{2} \int_M R_\gamma d \sigma_\gamma=A-Q\leq C_3 V.
\end{equation}
\end{proposition}
\begin{proof}
Gursky--Han's identity \eqref{eq:GH-identity} gives
\begin{equation}\label{eq:N3-GH}
 \int_MR_\gamma\,d\sigma_\gamma=2A-2Q.
\end{equation}
which implies \eqref{eq:N3-compensated}.
Recall from \eqref{eq:K-N} and \eqref{eq:def of H-N}, 
\[
K_3(r)=\int_r^1\sqrt{1-t^2}\,dt,
 \qquad
 H_3(r)=\frac{K_3(r)}{(1-r^2)^{3/2}}.
\]
From \eqref{eq: def of test function f-N}, we know that the $f_3(r)$ is defined by
\begin{equation}\label{eq:N3-f}
 f_3(r)=\frac{C_3}{3}\left(\frac1{C_3}-H_3(r)\right).
\end{equation}
Then $f_3(0)=0$, and using $H_3(1)=\frac{1}{3}$,
\begin{equation}\label{eq:N3-f1}
 f_3(1)=\frac13\left(1-\frac{C_3}{3}\right),
 \qquad
 3f_3(1)=1-\frac4{3\pi}<1.
\end{equation}
Moreover $0\le f_3\le f_3(1)$.  Indeed,
\begin{equation}\label{eq: H'}
 H_3'(r)=\frac{3rH_3(r)-1}{1-r^2}.
\end{equation}
For $r>0$,
\[K_3(r) \leq \frac{1}{r} \int_r^1 t\left(1-t^2\right)^{1 / 2} d t=\frac{\left(1-r^2\right)^{3 / 2}}{3 r},\]
so $H_3'(r) \leq 0.$ Thus $H_3(r) \leq H_3(0)=1 / C_3$. The lower bound $H_3(r) \geq 1 / 3$ follows from
\[K_3(r)=\frac{1}{2} \int_0^{1-r^2} u^{1 / 2}(1-u)^{-1 / 2} d u \geq \frac{1}{2} \int_0^{1-r^2} u^{1 / 2} d u=\frac{\left(1-r^2\right)^{3 / 2}}{3}.\]
These two bounds for $K_3(r)$  imply $0\le f_3\le f_3(1)$.
Using the identity \eqref{eq: H'} for $H_3'$ and differentiating once more, we have
\begin{equation}\label{eq:N3-Lf}
L_3 f_3=-\left(1-r^2\right) f_3^{\prime \prime}+3 r f_3^{\prime}=\beta_3.
\end{equation}
We next prove the positivity of the correction term
\begin{equation}\label{eq:N3-P-positive}
P_3[f_3](r)=f_3^{\prime}(r)+\frac{2 r}{3} f_3^{\prime \prime}(r).
\end{equation}
Indeed, by direct calculation,
\[
 P_3[f_3](r)=\frac{C_3}{9(1-r^2)^2}
 \left[3+7r^2-15r(1+r^2)H_3(r)\right].
\]
Thus, it is enough to show that
\begin{equation}\label{eq: H bound 1}
H_3(r)\le\frac{3+7r^2}{15r(1+r^2)} \qquad (0<r<1).
\end{equation}
Set
\[
 G_3(r)=\frac{(3+7r^2)(1-r^2)^{3/2}}{15r(1+r^2)}-K_3(r).
\]
Then, $G_3(1)=0$ and
\[
 G_3'(r)=
 \frac{\sqrt{1-r}\,(r-1)^2(1+r)^{5/2}(r^2-3)}
 {15r^2(1+r^2)^2}<0.
\]
Hence $G_3(r)\ge0$ on $(0,1)$, which proves \eqref{eq: H bound 1} and therefore implies the positivity of \eqref{eq:N3-P-positive}.

Combining \eqref{eq: Detla f rho main}, \eqref{eq:N3-Lf}, and positivity of \eqref{eq:N3-P-positive}, we obtain
\begin{align*}
 -\Delta_g f_3(\rho)=\beta_3(\rho)+wP_3[f_3](\rho)\ge\beta_3(\rho).
\end{align*}
Therefore, by the defect formula \eqref{eq:defect-D},
\[
 A-C_3V\le\int_Xw(-\Delta_gf_3(\rho))\,dv_g.
\]
Since $w=0$ and $f(\rho)=f(0)=0$ on $M$, the boundary terms vanish. Hence, integrating by parts together with \eqref{eq:Delta-U} implies
\[
 \int_Xw(-\Delta_gf_3(\rho))\,dv_g
 =\int_Xf_3(\rho)(-\Delta_gw)\,dv_g
 =3\int_Xf_3(\rho)\rho |E|^2\,dv_g.
\]
Using \eqref{eq:N3-f1}, we have
\[
 A-C_3V\le3f_3(1)Q<Q.
\]
This proves \eqref{eq:N3-compensated}. 
\end{proof}
\subsection{The four-dimensional case}
Let $N=4$ and
\[
 C_4=\frac{\area(\Sph^3)}{\vol(\Sph^4_+)}=\frac32.
\]
\begin{proposition}\label{lem:N4-compensated}
There hold
\begin{equation}\label{eq:N4-compensated}
 \frac{1}{6} \int_M R_\gamma d \sigma_\gamma=A-\frac{1}{6} Q \leq C_4 V
\end{equation}
and then
\begin{equation}\label{eq:N4-compensated-new}
\frac{1}{6} \int_M R_\gamma d \sigma_\gamma  \leq \left (\frac{C_4 V}{A}\right)^{\frac{2}{3}}.
\end{equation}
\end{proposition}

\begin{proof}
Simplify $\beta_N$ in \eqref{eq:beta-N} and $H_N$ in \eqref{eq:def of H-N} by 
\begin{equation}\label{eq:N4-beta}
 \beta_4(r)=\frac{r^2+3r+4}{4(1+r)^3}, \qquad H_4(r)=\frac{K_4}{(1-r^2)^{2}}=\frac{r+2}{3(r+1)^2}.
\end{equation}
By \eqref{eq: def of test function f-N} we have
\begin{equation}\label{eq:N4-f}
f_4(r)=\frac{C_4}{4}\left(\frac{1}{C_4}-H_4(r)\right)=\frac{r(2r+3)}{8(1+r)^2}.
\end{equation}
Then,
\[
 f_4(0)=0,
 \qquad
 0\le f_4(r)\le f_4(1)=\frac5{32}<\frac16,
\]
 where the monotonicity follows from $f_4'(r)>0$ below. Direct calculation yields
\[
 f_4'(r)=\frac{r+3}{8(1+r)^3},
 \qquad
 f_4''(r)= -\frac{r+4}{4(1+r)^4},
\]
and hence
\begin{equation}\label{eq:N4-Lf}
 -(1-r^2)f_4''(r)+4rf_4'(r)=\beta_4(r),
\end{equation}
and
\begin{equation}\label{eq:N4-P}
 P_4[f_4](r)=f_4'(r)+\frac r2f_4''(r)=\frac3{8(1+r)^4}\ge0.
\end{equation}
Thus by \eqref{eq: Detla f rho main}, \eqref{eq:N4-Lf} and \eqref{eq:N4-P},
\[
 -\Delta_gf_4(\rho)=\beta_4(\rho)+wP_4[f_4](\rho)\ge\beta_4(\rho).
\]
Therefore, using \eqref{eq:defect-D} and integrating by parts as in the proof of Proposition \ref{lem:N3-compensated} yields
\[
 \begin{aligned}
A-C_4 V & \leq \int_X w\left(-\Delta_g f_4(\rho)\right) d v_g \\
& =\int_X f_4(\rho)\left(-\Delta_g w\right) d v_g \\
& =\int_X f_4(\rho) \rho|E|^2 d v_g \leq \frac{5}{32} Q<\frac{1}{6} Q.
\end{aligned}
\]
On the other hand,  Gursky-Han's  identity \eqref{eq:GH-identity} gives
\[
A-\frac{1}{6} Q=\frac{1}{6} \int_M R_\gamma d \sigma_\gamma,
\]
which proves \eqref{eq:N4-compensated}. 

Finally, \eqref{eq:N4-compensated-new} is a direct consequence of \eqref{eq:N4-compensated} and the fact that $C_4 V \leq A$ by Lemma \ref{lem:beta-positive}.
\end{proof}
\subsection{The higher-dimensional case}
Finally, we derive the following nonlinear defect estimate when $N\ge5$.
\begin{proposition}\label{prop:power-bridge}
For $N=n+1\ge5$, there holds
\begin{equation}\label{eq:power-bridge}
 \left(\frac{C_NV}{A}\right)^{2/n}
 \ge
 \frac{\displaystyle\int_MR_\gamma\,d\sigma_\gamma}{n(n-1)A}.
\end{equation}
\end{proposition}

The proof is based on a one-dimensional estimate for the correction function introduced in Lemma \ref{lem:endpoint-regularity}. Recall from \eqref{eq:def of H-N} that
\begin{equation}\label{eq:H-N}
 H_N(r)=\frac{K_N(r)}{(1-r^2)^{N/2}},
\end{equation}
so that
\begin{equation}\label{eq:H-ODE}
 H_N'(r)=\frac{NrH_N(r)-1}{1-r^2},
 \qquad
 H_N(0)=\frac1{C_N},
 \qquad
 \lim_{r\to1}H_N(r)=\frac1N.
\end{equation}
By \eqref{eq: def of test function f-N} we have
\begin{equation}\label{eq:high-f}
 f_N(r)=\frac{C_N}{N}\left(\frac1{C_N}-H_N(r)\right).
\end{equation}
\begin{lemma}\label{lem:high-1D}
For $N\ge5$, let $H_N$ and $f_N$ be the functions as in \eqref{eq:H-N} and \eqref{eq:high-f}. Set
\begin{equation}\label{eq:calP-def}
 P_N(r)=P_N\left[f_N\right](r)=f_N'(r)+\frac{2r}{N}f_N''(r).
\end{equation}
Then,
\begin{equation}\label{eq:high-f-bounds}
 0\le f_N(r)\le f_N(1)=\frac1N\left(1-\frac{C_N}{N}\right).
\end{equation}
Moreover, 
\begin{equation}\label{eq:high-Lf}
 L_Nf_N=\beta_N,
\end{equation}
where 
\[
 L_Nf=-(1-r^2)f''+Nrf',
\]
as in \eqref{eq: def of LN PN 1D}.
Furthermore, 
\begin{equation}\label{eq:high-pointwise}
\beta_N(r)^2\le2C_NP_N(r)\qquad \text{ for } 0\le r\le1.
\end{equation}
In particular, $P_N(r)>0$ on $[0,1]$.
\end{lemma}

\begin{proof}
We first prove \eqref{eq:high-f-bounds}. The lower bound estimate for $H_N$ follows from 
\[K_N(r)=\frac{1}{2} \int_0^{1-r^2} u^{N / 2-1}(1-u)^{-1 / 2} d u \geq \frac{1}{2} \int_0^{1-r^2} u^{N / 2-1} d u=\frac{\left(1-r^2\right)^{N / 2}}{N}.\]
Then, $H_N \ge 1/N$. For $r>0$,
\[
K_N(r) \leq \frac{1}{r} \int_r^1 t\left(1-t^2\right)^{(N-2) / 2} d t=\frac{\left(1-r^2\right)^{N / 2}}{N r}.
\]
By \eqref{eq:H-ODE},
\[
 H_N'(r)=\frac{NrH_N(r)-1}{1-r^2}\le0
\]
on $(0,1)$, and then $H_N(r)\le H_N(0)=1/C_N$.  Hence, $ 0\le f_N(r)\le f_N(1),$ which implies \eqref{eq:high-f-bounds}.

Differentiating \eqref{eq:high-f} and using \eqref{eq:H-ODE} give
\begin{equation}\label{eq:fprime-high-proof}
 f_N'(r)=\frac{C_N}{N}\frac{1-NrH_N(r)}{1-r^2}.
\end{equation}
Substituting \eqref{eq:H-ODE} into the derivative of \eqref{eq:fprime-high-proof} yields
\[
 L_Nf_N=-(1-r^2)f_N''+Nrf_N'
 =C_N\frac{(1+r^2)H_N(r)-2r/N}{1-r^2}.
\]
By \eqref{eq:beta-via-H}, the last expression is exactly $\beta_N$. This proves \eqref{eq:high-Lf}.

It remains to prove the pointwise inequality \eqref{eq:high-pointwise}.  A direct calculation from \eqref{eq:H-ODE} and \eqref{eq:fprime-high-proof} gives 
\begin{equation}\label{eq:calP-explicit}
 \frac{P_N(r)}{C_N}
 =\frac{N+(N+4)r^2-N(N+2)r(1+r^2)H_N(r)}{N^2(1-r^2)^2}.
\end{equation}
Using 
\[
\frac{\beta_N(r)}{C_N}=\frac{\left(1+r^2\right) H_N(r)-2 r / N}{1-r^2},
\]
which is \eqref{eq:beta-via-H}, together with \eqref{eq:calP-explicit}, we know that
$\beta_N^2\le2C_NP_N$ is equivalent to
\begin{equation}\label{eq:H-upper-needed}
 \left((1+r^2)H_N(r)+r\right)^2
 \le
 r^2+\frac{2(1+r^2)}{N}+\frac{4r^2}{N^2}.
\end{equation}
Define
\begin{equation}\label{eq:Hplus}
 H_N^+(r)=
 \frac{-r+\left(r^2+\frac{2(1+r^2)}{N}+\frac{4r^2}{N^2}\right)^{1/2}}
 {1+r^2}.
\end{equation}
Then, \eqref{eq:H-upper-needed} is equivalent to
\begin{equation}\label{eq:HleHplus}
 H_N(r)\le H_N^+(r).
\end{equation}

We now prove this comparison inequality.

Set
\[
 B_N(r)=2N+(N^2+2N+4)r^2.
\]
Then
\[
 H_N^+(r)=\frac{-r+N^{-1}\sqrt{B_N(r)}}{1+r^2}.
\]
Direct differentiation gives
\begin{equation}\label{eq:Hplus-residual}
 (H_N^+)'-\frac{NrH_N^+-1}{1-r^2}
 =-\frac{r\,Z_N(r)}{N(1-r^2)(1+r^2)^2\sqrt{B_N(r)}},
\end{equation}
where
\begin{align}
 \mathcal A_N(r)
 &=(N^3+N^2+2N-4)r^4\label{eq:AcalN}\\ 
 &\qquad+(N^3+6N^2+4N+8)r^2
 +(N^2+2N-4),\\
 Z_N(r)
 &=\mathcal A_N(r)
 -r\sqrt{B_N(r)}\,\bigl(N^2(1+r^2)+4N\bigr).\label{eq:XiN}
\end{align}
We claim that $Z_N(r) \geq 0$ for $0 \leq r \leq 1$.
The polynomial $\mathcal A_N$ is strictly positive on $[0,1]$ for $N\ge5$.  Moreover $N^2(1+r^2)+4N>0$. Squaring the two nonnegative terms in \eqref{eq:XiN}, we obtain 
\begin{equation}\label{eq:Sidentity}
 A_N(r)^2
 -r^2B_N(r)\bigl(N^2(1+r^2)+4N\bigr)^2
=(1-r^2)^2\mathcal S_N(r),
\end{equation}
where
\begin{align}
 \mathcal S_N(r)
 &=(N^4-4N^3-4N^2-16N+16)r^4 \notag\\
 &\quad +(2N^4-24N^2-32N-32)r^2
 \\
 &\qquad+N^4+4N^3-4N^2-16N+16.\label{eq:S-explicit}
\end{align}
For $N\ge6$, all three coefficients in \eqref{eq:S-explicit} are positive.  For $N=5$,
\[
 \mathcal S_5(r)=-39r^4+458r^2+961,
\]
which is also positive for $0\le r\le1$.  Hence $\mathcal S_N(r)\ge0$ for all $N\ge5$.  Since $\mathcal A_N\ge0$ and the second term in \eqref{eq:XiN} is also nonnegative, \eqref{eq:Sidentity} implies
\[
 \mathcal A_N(r)
 \ge r\sqrt{B_N(r)}\,\bigl(N^2(1+r^2)+4N\bigr),
\]
and therefore $Z_N(r)\ge0$.  Thus by \eqref{eq:Hplus-residual},
\begin{equation}\label{eq:Hplus-subsolution}
 (H_N^+)'-\frac{NrH_N^+-1}{1-r^2}\le0.
\end{equation}
Finally set $z=H_N^+-H_N$.  Combining \eqref{eq:H-ODE} and \eqref{eq:Hplus-subsolution} gives
\[
 z'-\frac{Nr}{1-r^2}z\le0.
\]
Multiplying by the integrating factor $(1-r^2)^{N/2}$, we obtain
\[
 \frac{d}{dr}\Big((1-r^2)^{N/2}z(r)\Big)\le0.
\]
Both $H_N$ and $H_N^+$ extend continuously to $r=1$ and satisfy
\[
 H_N(1)=H_N^+(1)=\frac1N.
\]
Thus, for $0<r<R<1$, integrate from $r$ to $R$ and then let $R\to 1$.  Since $z$ is bounded at $1$,
\[
 \lim_{R\to 1}(1-R^2)^{N/2}z(R)=0.
\]
Thus $(1-r^2)^{N/2}z(r)\ge0$, hence $z(r)\ge0$ on $[0,1)$. $z(1)\ge 0$ follows by continuity.  This proves \eqref{eq:HleHplus}, hence \eqref{eq:H-upper-needed}, and therefore \eqref{eq:high-pointwise}.  

Finally, Lemma \ref{lem:beta-positive} gives $\beta_N>0$ on $[0,1]$.  Therefore \eqref{eq:high-pointwise} implies $P_N>0$ on $[0,1]$.
\end{proof}

\begin{proof}[Proof of Proposition \ref{prop:power-bridge}]
By \eqref{eq:defect-D}, the defect is
\[
 D=A-C_NV=\int_Xw\beta_N(\rho)\,dv_g.
\]
By Lemma \ref{lem:beta-positive}, we have $D\ge0$.  Set
\begin{equation}\label{eq:s-def}
 s=\frac{C_NV}{A}\in(0,1].
\end{equation}
Using \eqref{eq: Detla f rho main} and Lemma \ref{lem:high-1D}, we obtain
\begin{equation}\label{eq:high-Delta-f}
 -\Delta_gf_N(\rho)=\beta_N(\rho)+wP_N(\rho),
\end{equation}
where $f_N$ is the function defined in \eqref{eq:high-f}.
Therefore, by the defect formula \eqref{eq:defect-D} and integration by parts,
\begin{align*}
 D+\int_XP_N(\rho)w^2\,dv_g
 &=\int_Xw(-\Delta_gf_N(\rho))\,dv_g\\
 &=\int_Xf_N(\rho)(-\Delta_gw)\,dv_g\\
 &=\frac{N}{(N-2)^2}\int_Xf_N(\rho)\rho |E|^2\,dv_g.
\end{align*}
Since
\[f_N(\rho) \leq f_N(1)=\frac{1}{N}\left(1-\frac{C_N}{N}\right),\]
we obtain 
\begin{equation}\label{eq: high-pointwise D}
D+\int_X P_N(\rho) w^2 d v_g \leq \frac{1-N^{-1}C_N}{(N-2)^2} Q.
\end{equation}
Let
\[
 \lambda_N=1-\frac{C_N}{N}\in(0,1).
\]
We note that \(0<C_N<N\). Indeed,
\[
  C_N^{-1}
  =
  K_N(0)
  =
  \frac12\int_0^1 u^{N/2-1}(1-u)^{-1/2}\,du
  >
  \frac12\int_0^1 u^{N/2-1}\,du
  =
  \frac1N.
\]
Then \(0<\lambda_N=1-C_N/N<1\).

Then by \eqref{eq: high-pointwise D} we have
\begin{equation}\label{eq:Q-D-B}
 Q\ge\frac{(N-2)^2}{\lambda_N}
 \left(D+\int_XP_N(\rho)w^2\,dv_g\right).
\end{equation}
On the other hand, by Cauchy inequality and \eqref{eq:high-pointwise} in Lemma \ref{lem:high-1D},
\begin{align*}
 D^2
 &=\left(\int_Xw\beta_N(\rho)\,dv_g\right)^2\\
 &\le\left(\int_XP_N(\rho)w^2\,dv_g\right)
    \left(\int_X\frac{\beta_N(\rho)^2}{P_N(\rho)}\,dv_g\right)\\
 &\le 2C_NV\int_XP_N(\rho)w^2\,dv_g.
\end{align*}
Thus
\begin{equation}\label{eq:B-lower}
 \int_XP_N(\rho)w^2\,dv_g\ge\frac{D^2}{2C_NV}.
\end{equation}
Combining \eqref{eq:Q-D-B} and \eqref{eq:B-lower}, and writing $D=A(1-s)$, $C_NV=As$ for $s\in (0, 1]$ defined in \eqref{eq:s-def}, we obtain
\begin{equation}\label{eq:Q-lower-x}
 \frac{Q}{A}
 \ge
 \frac{(N-2)^2}{2\lambda_N}\left(\frac1s-s\right).
\end{equation}
Since $N=n+1$, $N-2=n-1$.  Therefore
\begin{equation}\label{eq:delta-lower}
 \frac{2Q}{n(n-1)^2A}
 \ge
 \frac1{n\lambda_N}\left(\frac1s-s\right)
 \ge
 \frac1n\left(\frac1s-s\right).
\end{equation}
For $0<x\le1$ and $n\ge2$, we claim
\begin{equation}\label{eq:elementary-x}
 s^{2/n}\ge1-\frac1n\left(\frac1s-s\right).
\end{equation}
Indeed, after writing $s=e^{-t}$ with $t\geq 0$, \eqref{eq:elementary-x} becomes
\[
 e^{-2t/n}+\frac2n\sinh t\ge1,
\]
whose derivative is 
\[
\frac2n(\cosh t-e^{-2t/n})\ge0.
\]
Combining \eqref{eq:delta-lower} and \eqref{eq:elementary-x} yields
\[
 s^{2/n}\ge1-\frac{2Q}{n(n-1)^2A}.
\]
Finally, the Gursky--Han's identity \eqref{eq:GH-identity} yields
\[
\int_M R_\gamma d \sigma_\gamma=n(n-1) A-\frac{2}{n-1} Q.
\]
Dividing by $n(n-1)A$, we then have
\[
 \frac{\int_MR_\gamma\,d\sigma_\gamma}{n(n-1)A}=1-\frac{2Q}{n(n-1)^2A}.
\]
Since $s=C_N V/A$, this proves \eqref{eq:power-bridge}.
\end{proof}

\section{Proof of the main theorem}\label{sec:main-proof}

We prove Theorem \ref{thm:main} by combining Propositions \ref{lem:N3-compensated}-\ref{prop:power-bridge}. 
\subsection{The relative comparison inequality}
Recall from Section \ref{Sect:Introduction} that $N=n+1 \geq 3$, and the \textbf{type-I} Escobar-Yamabe compactification is normalized by
\[
R_g=N(N-1), \quad H_g=0.
\]
We write
\[
V=\operatorname{Vol}_g(X), \quad A=\operatorname{Vol}_g(M), \quad V_{+}=\operatorname{Vol}\left(\Sph_{+}^N\right), \quad A_{+}=\operatorname{Vol}\left(\Sph^n\right).
\]
Under this normalization,
\begin{equation}\label{eq: Yamabe with voulme}
\frac{Y_1(\barX, M,[g])}{Y_1\left(\Sph_{+}^N, \Sph^n,[g_{\Sph_{+}^N}]\right)}=\left(\frac{V}{V_{+}}\right)^{2 / N} .
\end{equation}
Therefore, to prove the inequality in Theorem \ref{thm:main}, it is enough to prove the volume lower bound
\begin{equation}\label{eq: volume lower bound}
V \geq V_{+}\left(\frac{Y(M,[h])}{Y\left(\Sph^n,\left[g_{\Sph^n}\right]\right)}\right)^{n / 2}.
\end{equation}
Indeed, raising \eqref{eq: volume lower bound} by the power $2 / N$ gives 
\[
\frac{Y_1(\barX, M,[g])}{Y_1\left(\Sph_{+}^N, \Sph^n,[g_{\Sph_{+}^N}]\right)} \geq\left(\frac{Y(M,[h])}{Y(\Sph^n,[g_{\Sph^n}])}\right)^{n/N} .
\]

By Propositions \ref{lem:N3-compensated}-\ref{prop:power-bridge}  we have
\[
 \left(\frac{C_NV}{A}\right)^{2/n}
 \ge
 \frac{\int_MR_\gamma\,d\sigma_\gamma}{n(n-1)A}.
\]
By definition of the Yamabe constant  we obtain
\[
 \int_MR_\gamma\,d\sigma_\gamma\ge Y(M,[h])A^{(n-2)/n}.
\]
Then,
\begin{equation}\label{eq: high D 2}
 \left(\frac{C_NV}{A}\right)^{2/n}
 \ge
 \frac{Y(M,[h])}{n(n-1)}A^{-2/n}.
\end{equation}
This implies
\begin{equation}\label{eq:high-CV-final}
 C_NV\ge\left(\frac{Y(M,[h])}{n(n-1)}\right)^{n/2}.
\end{equation}
Since
\[
 Y(\Sph^n)=n(n-1)A_+^{2/n},
\]
\eqref{eq:high-CV-final} together with $C_N=A_+/V_+$ gives
\begin{equation}\label{eq:high-volume-final}
 V\ge V_+\left(\frac{Y(M,[h])}{Y(\Sph^n, [g_{\Sph^n}])}\right)^{n/2}.
\end{equation}
Finally, by \eqref{eq: Yamabe with voulme} and \eqref{eq:high-volume-final} we obtain
\[
 \frac{Y_1(\barX,M,[g])}{Y_1(\Sph^N_+,\Sph^n,[g_{\Sph_+^N}])}
 =\left(\frac{V}{V_+}\right)^{2/N}
 \ge
 \left(\frac{Y(M,[h])}{Y(\Sph^n, [g_{\Sph^n}])}\right)^{n/N}.
\]
This proves the comparison inequality \eqref{eq:main-conjecture}.

\subsection{Rigidity}
We end  by proving the rigidity part. The standard hyperbolic space $(\mathbb{H}^N, g_{\mathbb{H}^N})$, together with its hemisphere compactification $(\Sph_+^N, g_{\Sph_+^N})$, clearly realizes the equality. So we only need to prove the converse.

Suppose the equality holds in \eqref{eq:main-conjecture}.  First consider $N=3$.  
 In the proof of Proposition \ref{lem:N3-compensated}, we actually obtain
\[
A-C_3 V \leq 3 f(1) Q, \quad 3 f(1)<1.
\]
Hence,
\[
C_3 V \geq A-3 f(1) Q=\frac{1}{2} \int_M R_\gamma d \sigma_\gamma+(1-3 f(1)) Q.
\]
Since $Y(M,[h])=\int_M R_\gamma d \sigma_\gamma$ in dimension two, the equality in the final comparison inequality \eqref{eq:main-conjecture} forces $Q=0$.

Next consider $N=4$. The proof of Proposition \ref{lem:N4-compensated} implies the stronger bound
\[
A-C_4 V \leq \frac{5}{32} Q.
\]
Since $\int_M R_\gamma d \sigma_\gamma=6 A-Q$, we have
\[
C_4 V \geq A-\frac{5}{32} Q=\frac{1}{6} \int_M R_\gamma d \sigma_\gamma+\frac{1}{96} Q.
\]
If $Q>0$, then the estimate \eqref{eq:N4-compensated} is strict, hence  \eqref{eq:main-conjecture} should be a strict inequality. Therefore the equality in \eqref{eq:main-conjecture} again forces $Q=0$.
 
Finally assume $N\ge5$. We recall the notation used in the proof of Proposition \ref{prop:power-bridge}:
\[
s=\frac{C_N V}{A} \in(0,1], \quad \lambda_N=1-\frac{C_N}{N} \in(0,1).
\]
From the proof of Proposition \ref{prop:power-bridge}, we obtain 
\begin{equation}\label{eq: rigidity high D 1}
\frac{2 Q}{n(n-1)^2 A} \geq \frac{1}{n \lambda_N}\left(\frac{1}{s}-s\right) \geq \frac{1}{n}\left(\frac{1}{s}-s\right),
\end{equation}
and 
\begin{equation}\label{eq: rigidity high D 2}
s^{2 / n} \geq 1-\frac{1}{n}\left(\frac{1}{s}-s\right).
\end{equation}
Combining \eqref{eq: rigidity high D 1} and \eqref{eq: rigidity high D 2} yields Proposition \ref{prop:power-bridge}. If the equality holds in the final comparison inequality \eqref{eq:main-conjecture}, then the equality must hold in Proposition \ref{prop:power-bridge} as well. Therefore, the equality must hold throughout the inequality chain
\[
s^{2 / n} \geq 1-\frac{1}{n}\left(\frac{1}{s}-s\right) \geq 1-\frac{2 Q}{n(n-1)^2 A}.
\]
However, since $\lambda_N \in (0,1)$, the two inequalities in \eqref{eq: rigidity high D 1} can both be equalities only if
\[
\frac{1}{s}-s=0.
\]
Thus $s=1$. Returning to the endpoint equality in Proposition
\ref{prop:power-bridge},
\[
  s^{2/n}
  =
  1-\frac{2Q}{n(n-1)^2A},
\]
we immediately obtain \(Q=0\).
Thus in all cases, the equality in the main inequality \eqref{eq:main-conjecture} forces 
\[
 Q=\int_X\rho |E|^2\,dv_g=0.
\]
Since $\rho>0$ in $X$, we have $E\equiv0$.  Therefore
\[
 \Ric_g=(N-1)g.
\]
Moreover, the equation for $w$ in \eqref{eq:Delta-U}  together with $E\equiv0$ implies $\Delta_g w=0$. Since $w|_M=0$,  we get $w\equiv0$. Thus,
\[
 \Delta\rho=-N\rho.
\]
\eqref{eq:E-hessian} then implies
\begin{equation}\label{eq:Obata}
\nabla^2\rho=-\rho g.
\end{equation}
We then need the following lemma.
\begin{lemma}\label{lem: reflection lemma}
Let $\left(X^N, g\right)$ be a smooth compact Einstein manifold with totally geodesic boundary $M= \partial X$. Let $(\hat{X}, \hat{g})$ be the double of $(X, g)$ across $M$. Then $\hat{g}$ is a smooth Einstein metric on the closed doubled manifold $\hat{X}$. 
\end{lemma}
\begin{proof}
In Fermi coordinates near $M$
\[
g=d t^2+h_t,
\]
where $t=d_g(\cdot, M) \in [0,\varepsilon)$ is the inward distance to $M$ and $h_t$ is a  family of smooth metrics on $M$. The second fundamental form of the hypersurface $M_{t_0}=\{t=t_0$\} is
\[
L_{t_0}=\frac{1}{2} \partial_t h_t|_{t=t_0}.
\]
Since the boundary $M$ is totally geodesic, we have
\[
L_0=0, \quad \text { hence }\left.\quad \partial_t h_t\right|_{t=0}=0.
\]
Now define
\[
\widehat{h}_t=\left\{\begin{array}{ll}
h_t, & t \geq 0, \\
h_{-t}, & t \leq 0,
\end{array} \quad \widehat{g}=d t^2+\widehat{h}_t\right..
\]
Because $h_t$ is smooth for $t \geq 0$ and $h_0^{\prime}=0$, Taylor expansion gives
\[
h_t=h_0+\frac{1}{2} h_0^{\prime \prime} t^2+O\left(t^3\right) .
\]
Therefore
\[
h_{|t|}=h_0+\frac{1}{2} h_0^{\prime \prime} t^2+O\left(|t|^3\right) .
\]
Thus the doubled metric is at least $C^{2, \alpha}$ across $M$ for every $0<\alpha<1$. 
Since $\hat{g} \in C^{2, \alpha}$, $\Ric_{\hat{g}}$ is $C^{0, \alpha}$. On each side of the doubled manifold,
\[
\operatorname{Ric}_{\hat{g}}=\lambda \hat{g}.
\]
Both sides are continuous up to the hypersurface $M$, so the same identity holds on the whole $\hat{X}$. Hence $\hat{g}$ is a $C^{2, \alpha}$ Einstein metric on the closed manifold $\hat{X}$.
Then by \cite[Theorem 5.2]{DeTurckKazdan1981}, the Einstein metric $\hat{g}$ is real analytic in harmonic coordinates. Therefore the doubled metric $\hat{g}$ is a smooth Einstein metric on $\hat{X}$.
\end{proof}
By Lemma \ref{lem: reflection lemma}, the doubled manifold $(\hat{X}, \hat{g})$ is a smooth closed Einstein manifold. Extending $\rho$ oddly across $M$ gives a nonconstant function on the doubled manifold $\hat{X}$, which is smooth away from the boundary $M$. From \cite[(3.9)]{GurskyHan}, we know that near $M$, $\rho$ has the expansion of the form
\[
\rho=r-\frac{1}{3} c_3 r^3+O\left(r^{3+\alpha}\right),
\]
where $r(x)=d_{g}(x, M)$ denotes the distance to the boundary $M$ and $c_3$ is a smooth function on $M$. Hence the odd extension of $\rho$ across $M$ yields a $C^{2,\alpha}$ function $\hat{\rho}$ on $\hat{X}$, which satisfies the same equation \eqref{eq:Obata}, i.e.
\[
 \nabla^2\hat{\rho}=-\hat{\rho} \hat{g}.
\]
Therefore, $\hat{\rho}$ also satisfies
\[
\Delta \hat{\rho} =-N \hat{\rho}.
\]
Hence, by standard elliptic regularity, we know that $\hat{\rho}$ is actually smooth.
By Obata's theorem \cite{Obata}, the double manifold $(\hat{X}, \hat{g})$ is isometric to the round sphere $\Sph^N$, with $\hat{\rho}$ being a first spherical harmonic on $\Sph^N$. Since $M=\{\hat{\rho}=0\}$, the hypersurface $M$ is the zero set of a first spherical harmonic $\hat{\rho}$, which is the great equator $\mathbb{S}^{N-1}$. Therefore, $X$ is one of the two hemispheres $(\Sph_+^N, g_{\Sph_+^N})$.

Therefore, $(X,g)$ is the hemisphere and $g_+=\rho^{-2}g$ is the standard hyperbolic metric.  This proves the rigidity and completes the proof of Theorem \ref{thm:main}.

\bibliographystyle{amsalpha}
\bibliography{conformal_filling}

@incollection{AndersonTopics,
  author    = {Anderson, Michael T.},
  title     = {Topics in conformally compact {Einstein} metrics},
  booktitle = {Perspectives in Riemannian Geometry},
  series    = {CRM Proceedings \& Lecture Notes},
  volume    = {40},
  pages     = {1--26},
  publisher = {American Mathematical Society},
  address   = {Providence, RI},
  year      = {2006}
}

@article{BrendleChen,
  author  = {Brendle, Simon and  Chen, Sophie},
  title   = {An existence theorem for the {Yamabe} problem on manifolds with boundary},
  journal = {Journal of the European Mathematical Society},
  volume  = {16},
  number  = {5},
  pages   = {991--1016},
  year    = {2014},
  doi     = {10.4171/JEMS/453}
}

@article{ChangGeCompactness,
  author  = {Chang, Sun-Yung Alice and Ge, Yuxin},
  title   = {Compactness of conformally compact {Einstein} manifolds in dimension {$4$}},
  journal = {Advances in Mathematics},
  volume  = {340},
  pages   = {588--652},
  year    = {2018},
  doi     = {10.1016/j.aim.2018.10.007}
}

@misc{ChangGeFilling,
  author        = {Chang, Sun-Yung Alice and Ge, Yuxin},
  title         = {On the problem of filling by a {Poincar{\'e}--Einstein} metric in dimension {$4$}},
  year          = {2025},
  eprint        = {2509.18430},
  archivePrefix = {arXiv},
  primaryClass  = {math.DG},
  note          = {preprint, arxiv:2509.18430}
}

@article{ChangGeQing,
  author  = {Chang, Sun-Yung Alice and Ge, Yuxin and Qing, Jie},
  title   = {Compactness of conformally compact {Einstein} {$4$}-manifolds {II}},
  journal = {Advances in Mathematics},
  volume  = {373},
  pages   = {107325},
  year    = {2020},
  doi     = {10.1016/j.aim.2020.107325}
}

@article{Escobar,
  author  = {Escobar, Jos{\'e} F.},
  title   = {The {Yamabe} problem on manifolds with boundary},
  journal = {Journal of Differential Geometry},
  volume  = {35},
  number  = {1},
  pages   = {21--84},
  year    = {1992},
  doi     = {10.4310/jdg/1214447805}
}

@article{GrahamLee,
  author  = {Graham, C. Robin and Lee, John M.},
  title   = {{Einstein} metrics with prescribed conformal infinity on the ball},
  journal = {Advances in Mathematics},
  volume  = {87},
  number  = {2},
  pages   = {186--225},
  year    = {1991},
  doi     = {10.1016/0001-8708(91)90071-E}
}

@article{GurskyHan,
  author  = {Gursky, Matthew J. and Han, Qing},
  title   = {Non-existence of {Poincar{\'e}--Einstein} manifolds with prescribed conformal infinity},
  journal = {Geometric and Functional Analysis},
  volume  = {27},
  number  = {4},
  pages   = {863--879},
  year    = {2017},
  doi     = {10.1007/s00039-017-0412-7},
}

@article{GurskyHanStolz,
  author  = {Gursky, Matthew J. and Han, Qing and Stolz, Stephan},
  title   = {An invariant related to the existence of conformally compact {Einstein} fillings},
  journal = {Transactions of the American Mathematical Society},
  volume  = {374},
  number  = {6},
  pages   = {4185--4205},
  year    = {2021},
  doi     = {10.1090/tran/8271},
}

@article{GurskySzekelyhidi,
  author  = {Gursky, Matthew J. and Sz{\'e}kelyhidi, G{\'a}bor},
  title   = {A local existence result for {Poincar{\'e}--Einstein} metrics},
  journal = {Advances in Mathematics},
  volume  = {361},
  pages   = {106912},
  year    = {2020},
  doi     = {10.1016/j.aim.2019.106912},
}

@article{LeeFredholm,
  author  = {Lee, John M.},
  title   = {Fredholm operators and {Einstein} metrics on conformally compact manifolds},
  journal = {Memoirs of the American Mathematical Society},
  volume  = {183},
  number  = {864},
  pages   = {vi+83},
  year    = {2006},
  doi     = {10.1090/memo/0864}
}

@article{LiQingShi,
  author  = {Li, Gang and Qing, Jie and Shi, Yuguang},
  title   = {Gap phenomena and curvature estimates for conformally compact {Einstein} manifolds},
  journal = {Transactions of the American Mathematical Society},
  volume  = {369},
  number  = {6},
  pages   = {4385--4413},
  year    = {2017},
  doi     = {10.1090/tran/6867},
}

@article{Obata,
  author  = {Obata, Morio},
  title   = {Certain conditions for a {Riemannian} manifold to be isometric with a sphere},
  journal = {Journal of the Mathematical Society of Japan},
  volume  = {14},
  number  = {3},
  pages   = {333--340},
  year    = {1962},
  doi     = {10.2969/jmsj/01430333}
}

@article{WangWang,
  author  = {Wang, Xiaodong and Wang, Zhixin},
  title   = {On a sharp inequality relating {Yamabe} invariants on a {Poincar{\'e}--Einstein} manifold},
  journal = {Proceedings of the American Mathematical Society},
  volume  = {150},
  number  = {11},
  pages   = {4923--4929},
  year    = {2022},
  doi     = {10.1090/proc/15967},
}

@misc{BrendleWang,
      title={A dimension descent scheme for the positive mass theorem in arbitrary dimension}, 
      author={S. Brendle and Y. Wang},
      year={2026},
      eprint={2604.08473},
    archivePrefix={arXiv},
      note={preprint, arxiv:2604.08473}, 
}

@article{ChenLaiWang,
  title={{Escobar-Yamabe} compactifications for {Poincar{\'e}-Einstein} manifolds and rigidity theorems},
  author={Chen, Xuezhang and Lai, Mijia and Wang, Fang},
  journal={Advances in Mathematics},
  volume={344},
  pages={681--714},
  year={2019},
  publisher={Elsevier},
  doi={10.1016/j.aim.2018.12.003},
  url={https://www.sciencedirect.com/science/article/pii/S0001870818304547}
}

@article{Witten,
    author = "Witten, Edward",
    title = "{Anti de Sitter space and holography}",
    journal = "Adv. Theor. Math. Phys.",
    volume = "2",
    pages = "253--291",
    year = "1998",
    eprint = "hep-th/9802150",
}

@inproceedings{Fefferman1985conformal,
  title={Conformal invariants},
  author={Fefferman, C. and Graham, C. R.},
  booktitle={{\'E}lie Cartan et les math{\'e}matiques d'aujourd'hui},
  pages={95--116},
  year={1985},
  series={Ast{\'e}risque},
  publisher={Soci{\'e}t{\'e} math{\'e}matique de France}
}

@article{Kichenassamy2004,
  title = {On a conjecture of {Fefferman and Graham}},
  author = {Kichenassamy, Satyanad},
  journal = {Advances in Mathematics},
  volume = {184},
  number = {2},
  pages = {268--288},
  year = {2004},
  issn = {0001-8708},
  doi = {https://doi.org},
  url = {https://www.sciencedirect.com/science/article/pii/S0001870803001452}
}

@article{Chang_2024,
  title={Perturbation compactness and uniqueness for a class of conformally compact {Einstein manifolds}},
  author={Chang, Sun-Yung A. and Ge, Yuxin and Jin, Xiaoshang and Qing, Jie},
  journal={Advances in Nonlinear Studies},
  volume={24},
  number={1},
  pages={247--278},
  year={2024},
  publisher={Walter de Gruyter GmbH}
}

@misc{lee2025rigidity,
  title         = {Rigidity of {P}oincar{\'e}-{E}instein manifolds with flat {E}uclidean conformal infinity}, 
  author        = {Lee, Sanghoon and Wang, Fang},
  year          = {2025},
  eprint        = {2503.06062},
  archivePrefix = {arXiv},
  primaryClass  = {math.DG},
  note          = {preprint, arxiv:2503.06062}
}

@article{wang2025sharp,
  title         = {A Sharp Inequality Relating {Yamabe} Invariants on Asymptotically {Poincare-{E}instein} Manifolds with a {Ricci} Curvature Lower Bound},
  author        = {Wang, Xiaodong and Wang, Zhixin},
  journal       = {Communications in Contemporary Mathematics},
  year          = {2025},
  doi           = {10.1142/S021919972350064X},
}

@article{wang2026lower,
  title     = {Lower Bounds for the Relative Volume of {P}oincar{\'e}-{E}instein Manifolds},
  author    = {Wang, Fang and Zhou, Huihuang},
  journal   = {The Journal of Geometric Analysis},
  volume    = {36},
  number    = {1},
  pages     = {32},
  year      = {2026},
  publisher = {Springer},
  doi       = {10.1007/s12220-025-02280-1},
}

@article{wang2023note,
  title     = {A Note on the Compactness of {P}oincar{\'e}--{E}instein Manifolds},
  author    = {Wang, Fang and Zhou, Huihuang},
  journal   = {Communications in Contemporary Mathematics},
  volume    = {25},
  number    = {05},
  pages     = {2250015},
  year      = {2023},
  publisher = {World Scientific},
  doi       = {10.1142/S0219199722500158},
}

@article{qing2003rigidity,
  title     = {On the Rigidity for Conformally Compact {E}instein Manifolds},
  author    = {Qing, Jie},
  journal   = {International Mathematics Research Notices},
  volume    = {2003},
  number    = {21},
  pages     = {1141--1153},
  year      = {2003},
  publisher = {Oxford University Press},
  doi       = {10.1155/S107379280320803X}
}

@article{Chang2025Poincare,
  title={On the {Poincar\'{e}-Einstein} Manifolds with Cylindrical Conformal Infinity},
  author={Chang, Sun-Yung Alice and Yang, Paul and Zhang, Ruobing},
  journal={arXiv:2509.20325},
  year={2025},
  url={https://arxiv.org/abs/2509.20325}
}

@article{DeTurckKazdan1981,
  title = {Some regularity theorems in {Riemannian} geometry},
  author = {DeTurck, Dennis M. and Kazdan, Jerry L.},
  journal = {Annales scientifiques de l'École Normale Supérieure},
  volume = {14},
  number = {3},
  pages = {249--260},
  year = {1981},
  publisher = {Elsevier}
}

\end{document}